\documentclass[a4paper,11pt,twoside]{amsart}
\usepackage[english]{babel}
\usepackage[utf8]{inputenc}

\usepackage[a4paper,inner=3cm,outer=3cm,top=4cm,bottom=4cm,pdftex]{geometry}
\usepackage{fancyhdr}
\usepackage{color}
\usepackage{bold-extra}
\usepackage{ mathrsfs }

\usepackage{comment}
\usepackage{graphics}
\usepackage{aliascnt}
\usepackage[pdftex,citecolor=green,linkcolor=red]{hyperref}

\usepackage{amsmath}
\usepackage{amsfonts}
\usepackage{amssymb}
\usepackage{amsthm}
\usepackage{comment}
\usepackage{mathtools}
\usepackage{enumerate}
\usepackage{enumitem} 

\newtheorem{theorem}{Theorem}[section]

\newtheorem{corollary}[theorem]{Corollary}

\newtheorem{lemma}[theorem]{Lemma}
\newtheorem{proposition}[theorem]{Proposition}
\theoremstyle{definition}
\newtheorem{definition}[theorem]{Definition}
\newtheorem{remark}[theorem]{Remark}
\newtheorem{example}[theorem]{Example}
\newtheorem{conjecture}[theorem]{Conjecture}
\numberwithin{equation}{section}

\newcommand\eps{\varepsilon}

\renewcommand{\Re}{\textnormal{Re}}

\newcommand\F{\mathbb{F}}
\newcommand\E{\mathbb{E}}

\newcommand\R{\mathbb{R}}
\newcommand\Z{\mathbb{Z}}

\newcommand\C{\mathbb{C}}

\newcommand{\CF}{\mathcal{F}}

\newcommand{\CB}{\mathcal{B}}

\newcommand{\wh}{\widehat}

\newcommand{\Mod}[1]{\ (\mathrm{mod}\ #1)}

\newcommand{\str}{\operatorname{str}}

\begin{document}
\title{The sum-product phenomenon for dense subsets of finite fields}

\author[Shao]{Xuancheng Shao}
\address{Department of Mathematics, University of Kentucky\\
Lexington, KY 40506\\
USA}
\email{xuancheng.shao@uky.edu}
\thanks{XS was supported by NSF grant DMS-2452462.}


\maketitle

\begin{abstract}
Let $\F_p$ be a finite field of prime order $p$ and let $A \subset \F_p$ be a subset. In the dense regime when $|A| \geq \alpha p$ for some $\alpha \in (0,1)$, we determine the optimal constant $f(\alpha)$ in the inequality
$$
\max(|A+A|, |A\cdot A|) \geq (f(\alpha) - o(1))p.
$$
The proof relies on a structural result for sumsets of dense subsets, established via a regularity lemma in general finite abelian groups.
\end{abstract}

\section{Introduction}

Let $\F_p$ be a finite field of prime order $p$ and let $A \subset \F_p$ be a subset. A central question in additive combinatorics asks how large the sumset $A+A = \{a+b: a,b\in A\}$ and the product set $A\cdot A = \{ab: a,b\in A\}$ must be relative to the size of $A$. The sum-product phenomenon, the principle that no subset of a ring can have both additive and multiplicative structure unless it is essentially a subring, lies at the heart of this question. For the finite field $\F_p$ which lacks nontrivial subrings, the phenomenon suggests that a nontrivial lower bound on $\max(|A+A|, |A\cdot A|)$ holds in general.

The systematic study of sum-product estimates in finite fields began with the fundamental work of Bourgain, Katz, and Tao \cite{BKT}, who established the bound
$$
\max(|A+A|, |A\cdot A|) \gg |A|^{1+c}
$$
for some absolute constant $c>0$, provided that $|A|$ is neither too large or too small. This result has since been refined and the constant $c$ has been made explicit and improved, mostly recently to $c=1/4$ in \cite{MS}, under mild assumptions on $|A|$. See also \cite{MPRRS, RRS, RSS} and the references therein for related results.

We will be interested in the regime when $|A|$ is large. In this direction, Garaev \cite{Garaev} proved that if $|A| > p^{2/3}$ then
\begin{equation}\label{eq:Garaev}
\max(|A+A|, |A\cdot A|) \gg (p|A|)^{1/2},
\end{equation}
which is sharp up to the implied constant. In this paper, we focus on the regime when $A$ is dense in $\F_p$. For fixed $\alpha \in (0,1)$, define
$$
f(\alpha) = \liminf_{p\rightarrow \infty} \min_{A \subset \F_p, |A| \geq \alpha p} \frac{\max(|A+A|, |A\cdot A|)}{p}.
$$
The estimate \eqref{eq:Garaev} implies that $f(\alpha) \gg \alpha^{1/2}$. Our main result determines the exact value of $f(\alpha)$ for every $\alpha$.

\begin{theorem}\label{thm:falpha}
For every $\alpha \in (0,1]$ we have
$$
f(\alpha) = \begin{cases} 1 & \text{if }\frac{1}{2} \leq \alpha \leq 1 \\
\displaystyle{\min_{\ell} \max\Big(2\ell \alpha, \frac{1}{\ell}\Big)}, & \text{if } \alpha \leq \frac{1}{2}, \end{cases}
$$
where the minimum above is taken over all positive integers $\ell$.
\end{theorem}

Equivalently, $f(\alpha)$ is the continuous function defined piecewise as follows:
$$
f(\alpha) = \begin{cases} 1 & \text{if }\frac{1}{2} \leq \alpha \leq 1, \\
\frac{1}{\ell} & \text{if }\frac{1}{2\ell(\ell+1)} \leq \alpha \leq \frac{1}{2\ell^2}\text{ for some positive integer }\ell, \\
2(\ell+1)\alpha & \text{if } \frac{1}{2(\ell+1)^2} \leq \alpha \leq \frac{1}{2\ell(\ell+1)}\text{ for some positive integer }\ell. \end{cases}
$$
In particular, we have $f(\alpha) = 1$ for $\alpha \geq 1/4$, implying that either $A+A$ or $A\cdot A$ must be almost all of $\F_p$, provided that the density of $A$ exceeds $1/4$. On the other hand, for small $\alpha$ one can show that the minimum over $\ell$ is achieved at some $\ell \sim (2\alpha)^{-1/2}$, leading to the estimate 
\begin{equation}\label{eq:small-alpha}
f(\alpha) = (1+o(1)) (2\alpha)^{1/2},
\end{equation}
where $o(1)$ denotes a quantity that tends to $0$ as $\alpha\rightarrow 0$.

\begin{example}\label{example:sharp}
The following construction shows that 
$$
f(\alpha) \leq \max(2\ell\alpha, \ell^{-1})
$$ 
for every $\alpha \in (0,1)$ and every $\ell$. We may assume that $\max(2\ell\alpha, \ell^{-1}) < 1$, so that $\ell \geq 2$ and $\alpha < 1/(2\ell)$. Let $p$ be a large prime with $\ell \mid p-1$, and let $H \subset \F_p^{\times}$ be the multiplicative subgroup of index $\ell$. Define $A = [1, N] \cap H$ for some positive integer $N < p$. Then, letting $\chi\Mod{p}$ be a Dirichlet character of order $\ell$, we have
$$
|A| = \frac{1}{\ell}\sum_{1 \leq n \leq N} \sum_{j=0}^{\ell-1} \chi^j(n) = \frac{N}{\ell} + O(p^{1/2}\log p)
$$
by the P\'{o}lya-Vinogradov inequality on character sums. Hence, we may choose $N = (\ell \alpha + o(1))p$ such that $|A| \geq \alpha p$. Then $|A+A| \leq (2\ell \alpha + o(1))p$ and $|A\cdot A| \leq |H| \leq p/\ell$, and hence
$$
\max(|A+A|, |A\cdot A|) \leq (\max(2\ell\alpha, \ell^{-1}) + o(1))p.
$$
This implies the desired upper bound on $f(\alpha)$.
\end{example}

In view of this construction, the main content of Theorem \ref{thm:falpha} is the lower bound for $f(\alpha)$, which will be discussed in Section \ref{sec:outline}. Motivated by Garaev's estimate \eqref{eq:Garaev} and the asymptotic \eqref{eq:small-alpha}, we propose the following conjecture.

\begin{conjecture}
Let $A \subset \F_p$ be a subset with $|A| = o(p)$ and $|A| \geq p^{1-c}$ for some sufficiently small constant $c>0$. Then
$$
\max(|A+A|, |A\cdot A|) \geq (1+o(1))(2p|A|)^{1/2}.
$$
\end{conjecture}

\begin{remark}
It is natural to ask what happens in finite fields $\F_q$ when $q$ is not a prime. While most of the machinery we develop in this paper continue to work in this general setting, the second part of Corollary \ref{cor:sumset} is no longer applicable to $G = \Z/q\Z$ when $q$ is not a prime (since its proof uses the Cauchy-Davenport theorem).

As a specific example, let $q = 2^{2n}$, let $V \subset \F_q$ be the additive subgroup of index $2$, and let $T$ be the coset of $V$ which does not contain $0$. Now let $H \subset \F_q^{\times}$ be the multiplicative subgroup of index $3$. The pigeonhole principle guarantees the existence of a coset $gH$ such that the set $A = T \cap gH$ satisfies $|A| \geq |T|/3 = q/6$. Since
$$
A + A \subset T+T = V\text{ and } A\cdot A \subset (gH)\cdot (gH) = g^2H,
$$
we have $\max(|A+A|,|A\cdot A|) \leq q/2$. On the other hand, $f(1/6)=2/3 > 1/2$. This shows that Theorem \ref{thm:falpha} fails in general finite fields $\F_q$.
\end{remark}

At the end of this introduction, we mention the natural variant of studying the size of $|A+A\cdot A|$ or $|A\cdot (A+A)|$. In \cite{AMRS,RRS} it is shown that both $A+A\cdot A$ and $A\cdot (A+A)$ must occupy a positive proportion of $\F_p$ as soon as $|A| \gg p^{2/3}$. These results exemplify the expander philosophy that algebraic operations saturate the field beyond a certain density threshold. In the dense regime when $|A| \gg p$, it follows from results in \cite{BHS,HH} that both $A+A\cdot A$ and $A\cdot (A+A)$ must occupy almost all of $\F_p$. Furthermore, the precise threshold for the density of $A$ in order for $A\cdot (A+A)$ to contain all of $\F_p^{\times}$ was recently resolved by Semchankau \cite{Semchankau}. Finally, the analogue $|A+A^{-1}|$ and its generalizations in the dense case $|A| \gg p$  have been studied in \cite{Semchankau,SS}.

\section{Proof outline of Theorem \ref{thm:falpha}}\label{sec:outline}

In this section, we outline the key ingredients in the proof of Theorem \ref{thm:falpha}. Recall the upper bound for $f(\alpha)$ demonstrated in Example \ref{example:sharp}. The lower bound for $f(\alpha)$ is a consequence of the following result.

\begin{theorem}\label{thm:main}
Let $\eps \in (0,1/2)$ and let $p$ be sufficiently large in terms of $\eps$. Let $A \subset \F_p$ be a subset such that
$$
\max(|A+A|, |A\cdot A|) \leq (\beta-\eps) p,
$$
for some $\beta < 1/\ell$ and some positive integer $\ell$. Then
$$
|A| \leq \frac{\beta}{2(\ell+1)}  p.
$$
\end{theorem}

\begin{proof}[Proof of Theorem \ref{thm:falpha} assuming Theorem \ref{thm:main}]
We may assume that $\alpha \leq 1/2$. Suppose, for the purpose of contradiction, that
$$
f(\alpha) < \min_{\ell} \max\Big(2\ell\alpha, \frac{1}{\ell}\Big)
$$
for some $\alpha \in (0,1)$. Then we may find $\eps \in (0,1/2)$, arbitrarily large $p$ and $A \subset \F_p$ with $|A| \geq\alpha p$ such that 
$$
\max(|A+A|, |A\cdot A|) \leq (\beta-\eps) p
$$
for some $\beta$ satisfying
$$
\beta < \min_{\ell} \max\Big(2\ell\alpha, \frac{1}{\ell}\Big).
$$ 
Since $\beta < 1$, Theorem \ref{thm:main} implies that $|A| \leq (\beta/4)p$, and hence $\alpha \leq \beta/4$. Now set $\ell = \lfloor \beta/(2\alpha)\rfloor \geq 2$. Since $\beta < \max(2\ell\alpha, \ell^{-1})$ and $\ell \leq \beta/(2\alpha)$, 
we  must have $\max(2\ell\alpha, \ell^{-1}) = \ell^{-1}$ and $\beta < \ell^{-1}$. Then Theorem \ref{thm:main} implies that 
$$
\alpha \leq \frac{\beta}{2(\ell+1)} \Longrightarrow \ell+1 \leq \frac{\beta}{2\alpha},
$$
contradicting our choice of $\ell$.
\end{proof}

Lying at the heart of our proof of Theorem \ref{thm:main} is a structural result for $A+A$, whose proof uses a weak regularity lemma. We now discuss these two ingredients and outline how Theorem \ref{thm:main} can be deduced from them.

\subsection{A weak regularity lemma}

In the broadest sense, a regularity lemma is a structure theorem stating that any  object can be decomposed into a bounded number of structured components and some small error terms. Perhaps the most famous example is Szemer\'{e}di's regularity lemma \cite{Szemeredi} from graph theory, which was used by Szemer\'{e}di to prove the existence of $k$-APs in dense subsets of integers. 

In the arithmetic setting, the arithmetic regularity lemma by Green and Tao \cite{GreenTao} has become a powerful tool to study linear patterns controlled by Gowers norms. The abelian case of it, when the linear pattern is controlled by Gowers $U^2$-norm, was used in \cite{EGM} and presented in the note \cite{Eberhard}. In this paper we follow the arguments in \cite{Eberhard, GreenTao} to generalize the abelian case of the arithmetic regularity lemma to general finite abelian groups.

Let $G$ be a finite abelian group with the group operation written additively. A character on $G$ is a homomorphism $\gamma: G\rightarrow \R/\Z$. Let $\wh{G}$ be the group of characters on $G$. The Fourier transform of a function $f: G\rightarrow \C$ is the function $\wh{f}: \wh{G}\rightarrow \C$ defined by
$$
\wh{f}(\gamma) = \E_{x \in G} f(x) e(-\gamma(x)),
$$
where $e(y) = e^{2\pi iy}$ for $y \in \R/\Z$. The $L^p$-norm of $f$ and the $L^p$-norm of $\wh{f}$ are defined by
$$
\|f\|_p^p = \E_x |f(x)|^p \ \ \text{and} \ \ \|\wh{f}\|_p^p = \sum_{\gamma} |\wh{f}(\gamma)|^p, 
$$
respectively. The Gowers $U^2$-norm of $f$ is defined by
$$
\|f\|_{U^2(G)}^4 = \E_{x,h_1,h_2 \in G} f(x) \overline{f(x+h_1)f(x+h_2)} f(x+h_1+h_2),
$$
and one can show by standard Fourier analysis that $\|f\|_{U^2(G)} = \|\wh{f}\|_4$. The convolution $f*g$ of two functions $f,g : G\rightarrow \C$ is defined by
$$
(f*g)(x) = \E_{y \in G} f(y) g(x-y).
$$
The inner products are defined by
$$
\langle f,g\rangle = \E_{x \in G} f(x) \overline{g(x)} \ \ \text{and} \ \ \langle\wh{f},\wh{g}\rangle = \sum_{\gamma} \wh{f}(\gamma)\overline{\wh{g}(\gamma)}.
$$
We say that $f$ is $1$-bounded if $\|f\|_{\infty} \leq 1$. A growth function is an increasing function $\CF: \R_{>0}\rightarrow \R_{>0}$.

\begin{definition}\label{def:1-measurable}
A $1$-bounded function $f: G\rightarrow \C$ is said to be $1$-measurable with growth $\CF$ if for every $M > 0$ there is some $1$-bounded function $f_{\str}:G\rightarrow \C$ with $\|\wh{f_{\str}}\|_1 \leq \CF(M)$ such that $\|f - f_{\str}\|_2 \leq M^{-1}$. A subset $E \subset G$ is said to be $1$-measurable with growth $\CF$ if $1_E$ is $1$-measurable with growth $\CF$.
\end{definition}

Our definition of $1$-measurable functions on $G$ is adapted from the definition in \cite{Eberhard} for functions on $[N] := \{1,2,\cdots,N\}$, which in turn is the special case of the definition of $s$-measurable functions on $[N]$ from \cite{GreenTao}. Just as $s$-measurable functions on $[N]$ are well approximated by $s$-step nilsequences with bounded complexity, $1$-measurable functions on $G$ are well approximated by linear combinations of a bounded number of characters on $G$.  Loosely speaking, a  $1$-measurable set has robust additive structure. Examples include arithmetic progressions when $G$ is cyclic, subspaces when $G = \F_2^n$, and (regular) Bohr sets for general $G$. 

\begin{definition}\label{def:1-factor}
A factor $\CB$ of $G$ is a partition of $G$ into cells. It is said to be a $1$-factor with complexity at most $M$ and growth $\CF$ if it has at most $M$ cells and each cell is $1$-measurable with growth $\CF$. 
\end{definition}

Examples of $1$-factors (with bounded complexity and controlled growth) include the partition of a cyclic group into a bounded number of arithmetic progressions, the partition of $G$ into cosets of a subgroup with bounded index and, more generally, the partition of $G$ into a bounded number of translates of Bohr sets.

For a function $f: G\rightarrow \C$ and a factor $\CB$ of $G$, the function $f_{\CB}: G\rightarrow \C$ is defined to be
\begin{equation}\label{eq:fB}
f_{\CB}(x) = \E_{y \in \CB(x)} f(y),
\end{equation}
where $\CB(x)$ denotes the cell of $\CB$ containing $x$. Thus $f_{\CB}$ is constant on each cell of $\CB$ and the value is equal to the average of $f$ over the cell.

\begin{proposition}[Weak regularity]\label{prop:weak-reg}
Let $f: G\rightarrow \C$ be a $1$-bounded function and let $\delta \in (0,1/2)$. There exists a growth function $\CF$ depending only on $\delta$ and a $1$-factor $\CB$ of $G$ with complexity $O_{\delta}(1)$ and growth $\CF$, such that
$$
\|f - f_{\CB}\|_{U^2(G)} \leq \delta.
$$
\end{proposition}

There are a number of other regularity lemmas of this type for subsets of a general finite abelian group $G$ or a cyclic group; see \cite{Green, GreenMorris, GreenRuzsa, Semchankau}. Proposition \ref{prop:weak-reg} is perhaps very similar in spirit with the regularity lemma in \cite{Green} (where smooth Bohr cutoffs are used instead of $1$-factors) and the wrapping lemma in \cite{Semchankau}. Proposition \ref{prop:weak-reg} will be proved in Section \ref{sec:regularity}.

\subsection{A structure theorem for sumsets}

A fundamental question in the inverse sumset theory is to describe the structure of $A$ or $A+A$ when $|A+A|$ is small. Using our regularity lemma, we prove the following structural result for $A+A$ for dense subsets $A \subset G$. For $\eps > 0$, we say that a sum $x \in A+A$ is \emph{$\eps$-popular} if $|A \cap (x-A)| \geq \eps |G|$, or equivalently if there are at least $\eps |G|$ ways to write $x$ as the sum of two  elements in $A$. Define $S_{\eps}(A)$ to be the set of $\eps$-popular sums in $A+A$:
\begin{equation}\label{eq:populuar}
S_{\eps}(A) = \{x \in G: |A \cap (x-A)| \geq \eps |G|\}.
\end{equation}

\begin{proposition}[Structure of sumsets]\label{prop:sumset}
Let $A \subset G$ be a subset and let $\eps \in (0,1/2)$. There exists a growth function $\CF$ depending only on $\eps$ and a $1$-measurable subset $A' \subset G$ with growth $\CF$, such that $|A\setminus A'| \leq \eps |G|$ and $|S_{\eps}(A') \setminus (A+A)| \leq \eps |G|$.
\end{proposition}

In loose terms, Proposition \ref{prop:sumset} produces a highly structured set $A'$ which contains almost all of $A$ (but could be a lot larger than $A$), such that $A'+A' \approx A+A$. A similar result can be proved using the wrapping lemma in \cite{Semchankau} when $G$ is cyclic of prime order. It is also worth mentioning the work of Croot and Sisask \cite{CrootSisask} who proved that $A+A$ has a large number of almost periods, a very general result that has a number of applications and also a key ingredient in the arguments in \cite{Semchankau}.

Once equipped with Proposition \ref{prop:sumset}, the proof of Theorem \ref{thm:main} roughly proceeds as follows. Apply Proposition \ref{prop:sumset} to $A$ as a subset of the additive group $\F_p$ and then as a subset of the multiplicative group $\F_p^{\times}$ to produce two sets $A_1,A_2$ containing almost all of $A$, such that $A_1$ has robust additive structure and $A_2$ has robust multiplicative structure, and
$$
A_1+A_1 \approx A+A, \ \ A_2\cdot A_2 \approx A\cdot A.
$$
Hence the hypothesis
$$
\max(|A+A|, |A\cdot A|) \leq (\beta - \eps)p
$$
for some $\beta < 1/\ell$ implies the similar upper bound for $|A_1+A_1|$ and $|A_2\cdot A_2|$. Using the Cauchy-Davenport theorem in the additive group $\F_p$ and Kneser's theorem in the multiplicative group $\F_p^{\times}$, one can show that
$$
|A_1| \leq \frac{\beta}{2}p \ \ \text{and} \ \ |A_2| \leq \frac{1}{\ell+1}p.
$$
Finally, using the additive structure of $A_1$ and the multiplicative structure of $A_2$, we will show that
$$
|A_1 \cap A_2| \sim \frac{|A_1||A_2|}{p} \leq \frac{\beta}{2(\ell+1)}p,
$$
leading to the desired upper bound for $A$ since $A$ is essentially contained in $A_1 \cap A_2$.

The deduction of Theorem \ref{thm:main} from Proposition \ref{prop:sumset} as outlined above will be carried out in Section \ref{sec:deduction}, after Section \ref{sec:sumset} where we prove Proposition \ref{prop:sumset} using the regularity lemma.

\section{A weak regularity lemma}\label{sec:regularity}

In this section we prove Proposition \ref{prop:weak-reg}. Let $G$ be a finite abelian group. Recall the notion of $1$-measurable functions and $1$-measurable subsets from Definition \ref{def:1-measurable} and the notion of $1$-factors from Definition \ref{def:1-factor}. 

\begin{lemma}\label{lem:intersect}
Let $f,g:G\rightarrow \C$ be $1$-bounded functions which are $1$-measurable with growth $\CF$. Then there exists a growth function $\CF'$ depending only on $\CF$ such that $fg$ is $1$-measurable with growth $\CF'$.
\end{lemma}

\begin{proof}
For any $M > 0$, let $f_{\str},g_{\str}:G\rightarrow \C$ be $1$-bounded functions with $\|\wh{f_{\str}}\|_1 \leq \CF(M)$ and $\|\wh{g_{\str}}\|_1 \leq \CF(M)$, such that  $\|f - f_{\str}\|_2 \leq M^{-1}$ and  $\|g - g_{\str}\|_2 \leq M^{-1}$. Then the product $f_{\str}g_{\str}$ satisfies
$$
\|\wh{f_{\str}g_{\str}}\|_1 = \|\wh{f_{\str}} * \wh{g_{\str}}\|_1 \leq \|\wh{f_{\str}}\|_1 \|\wh{g_{\str}}\|_1 \leq \CF(M)^2
$$
and
$$
\|fg - f_{\str}g_{\str}\|_2 \leq \|f-f_{\str}\|_{L^2} + \|g-g_{\str}\|_{L^2} \leq 2 M^{-1}.
$$
The conclusion follows by taking $\CF'(M) = \CF(2M)^2$.
\end{proof}

In particular, Lemma \ref{lem:intersect} implies that if $E_1,E_2\subset G$ are $1$-measurable with growth $\CF$, then $E_1 \cap E_2$ is $1$-measurable with growth $\CF'$ for some $\CF'$ depending only on $\CF$.

\begin{lemma}\label{lem:HL-maximal}
Let $X \subset \R/\Z$ be a finite set. For $t \in \R/\Z$, define
$$
M(t) = \sup_{r > 0} \frac{|X \cap (t-r,t+r)|}{2r|X|}.
$$
Then for any $\lambda > 0$ we have
$$
\mu(\{t \in \R/\Z: M(t) > \lambda\}) \ll \lambda^{-1},
$$
where $\mu(\cdot)$ denotes the Lebesgue measure.
\end{lemma}

\begin{proof}
This follows from the Hardy-Littlewood maximal inequality. For completeness, we include the proof.  For each $t \in T := \{t \in \R/\Z: M(t) > \lambda\}$, choose $I_t = (t-r, t+r)$ for some $r > 0$ such that
$$
|X \cap I_t| > \lambda\cdot \mu(I_t)|X|.
$$
By the Besicovitch covering lemma applied to the cover of $T$ by $\{I_t: t\in T\}$, there exists a subset $S\subset T$ such that $\{I_s: s \in S\}$ still covers $T$ and each $t \in \R/\Z$ lies in at most $O(1)$ intervals $I_s$ with $s \in S$.  Hence
$$
\mu(T) \leq \sum_{s \in S} \mu(I_s) < \lambda^{-1}\sum_{s \in S} \frac{|X\cap I_s|}{|X|}= \lambda^{-1}\frac{1}{|X|} \sum_{x \in X} \sum_{s \in S} 1_{x \in I_s}.
$$
For each $x \in X$, the inner sum over $s \in S$ above is $O(1)$, and thus $\mu(T) \ll \lambda^{-1}$, completing the proof.
\end{proof}

The starting point of the proof of Theorem \ref{prop:weak-reg} is the following inverse theorem, which states that if a function has large $U^2$-norm then it must correlate with a $1$-measurable subset.

\begin{lemma}\label{lem:u2-inverse-E}
Let $f: G\rightarrow \C$ be a $1$-bounded function. If $\|f\|_{U^2(G)} \geq \delta$ for some $\delta \in (0,1/2)$, then there exists a growth function $\CF$ depending only on $\delta$ and  a $1$-measurable subset $E \subset G$ with growth $\CF$ such that
$$
\Big|\E_{x \in G} f(x) 1_E(x)\Big| \gg \delta^4.
$$
\end{lemma}

\begin{proof}
Since $\|f\|_{U^2(G)} = \|\wh{f}\|_4$, we have
$$
\delta^4 \leq \sum_{\gamma} |\wh{f}(\gamma)|^4 \leq \sup_{\gamma} |\wh{f}(\gamma)|^2.
$$
Hence there exists $\gamma \in \wh{G}$ such that
$$
\Big|\E_{x \in G} f(x) e(\gamma(x))\Big| \geq \delta^2.
$$
For $t \in \R/\Z$, let $E_t = \{x \in G: \gamma(x) \in [t, t + \ell]\}$ for some $\ell > 0$ sufficiently small in terms of $\delta$.  Since $e(\gamma(x)) = e(t)+ O (\ell)$ for $x \in E_t$, we have
$$
\ell^{-1}\int_0^1 e(t) 1_{E_t}(x) dt = e(\gamma(x)) + O(\ell)
$$
for all $x \in G$. By choosing $\ell = c\delta^2$ for some sufficiently small absolute constant $c>0$, we may ensure that the error $O(\ell)$ above is negligible and obtain
$$
\Big|\int_0^1 \Big(\E_{x \in G} f(x) 1_{E_t}(x)\Big) e(t) dt \Big| \gg \delta^2\ell \gg \delta^4.
$$
Thus there exists a subset $T \subset \R/\Z$ with Lebesgue measure $\mu(T) \gg \delta^4$ such that
$$
\Big|\E_{x \in G} f(x) 1_{E_t}(x)\Big| \gg \delta^4
$$
for all $t \in T$. Define
$$
M(t) = \sup_{r > 0} \frac{|\{x \in G: \gamma(x) \in (t-r,t+r)\}|}{2r|G|}.
$$
By Lemma \ref{lem:HL-maximal}, there exists $t \in T$ such that $M(t) \ll \delta^{-4}$ and $M(t+\ell) \ll \delta^{-4}$. For such $t$, we show that $E_t$ is $1$-measurable with growth $\CF$ for some growth function $\CF$ depending only on $\delta$.

For $r\in (0,1/10)$, let $\Phi_r: \R/\Z \rightarrow \R_{\geq 0}$ be a  bump function supported on $[-r,r]$ with $\int\Phi_r=1$, so that 
$$
|\wh{\Phi_r}(k)| \ll \min\Big(1, \frac{1}{(r|k|)^2}\Big)
$$
for all $k \in \Z$, where the Fourier transform $\wh{\Phi}_r: \Z\rightarrow \C$ is defined by
$$
\wh{\Phi_r}(k) = \int_0^1 \Phi_r(x) e(-kx) dx.
$$
(For example, $\Phi_r = r^{-2}1_{[-r/2,r/2]}*1_{[-r/2,r/2]}$ does the job.) Define $f_r: G\rightarrow \C$ by 
$$
f_r(x) = (1_{[t,t+\ell]}*\Phi_r)(\gamma(x)).
$$ 
Since $1_{[t,t+\ell]}*\Phi_r$ is supported on $[t-r,t+\ell+r]$ and equals to $1$ on $[t+r, t+\ell-r]$, $f_r(x)$ agrees with $1_{E_t}(x)$ unless $\gamma(x) \in [t-r, t+r] \cup [t+\ell-r, t+\ell+r]$. Since the number of such $x \in G$ is  at most
$$
2r|G| (M(t) + M(t+\ell)) \ll \delta^{-4} r|G|,
$$
we have $\|1_{E_t} - f_r\|_2 \ll \delta^{-2} r^{1/2}$. On the other hand, using the Fourier expansion of $1_{[t,t+\ell]}*\Phi_r$, we can write
$$
f_r(x) = \sum_{k \in \Z} \wh{1_{[t,t+\ell]}}(k) \wh{\Phi_r}(k) e(k\gamma(x)).
$$
It follows that
$$
\|\wh{f_r}\|_1 \leq \sum_{k \in \Z} |\wh{1_{[t,t+\ell]}}(k) \wh{\Phi_r}(k)| \leq \sum_{k \in \Z} |\wh{\Phi_r}(k)| \ll r^{-1},
$$
where the last inequality follows from the Fourier decay estimate for $\Phi_r$. This shows that $E_t$ is $1$-measurable with growth $\CF$, where $\CF(M) = c(\delta)M^{-2}$ for some constant $c(\delta)>0$.
\end{proof}

We will repeatedly apply Lemma \ref{lem:u2-inverse-E} to construct a $1$-factor $\CB$ such that $f-f_{\CB}$ has small $U^2$-norm. We say that a factor $\CB'$ refines another factor $\CB$ if every cell of $\CB$ is a union of cells of $\CB'$. Recall the definition of $f_{\CB}$ from \eqref{eq:fB}.

\begin{lemma}\label{lem:factor-pyth}
Let $\CB',\CB$ be factors of $G$ such that $\CB'$ refines $\CB$. Then for any function $f: G\rightarrow \C$ we have
$$
\|f_{\CB'}\|_2^2 - \|f_{\CB}\|_2^2 = \|f_{\CB'}-f_{\CB}\|_2^2.
$$
\end{lemma}

\begin{proof}
We have
$$
\|f_{\CB'}-f_{\CB}\|_2^2 = \langle f_{\CB'}-f_{\CB}, f_{\CB'}-f_{\CB}\rangle = \|f_{\CB'}\|_2^2 + \|f_{\CB}\|_2^2 - 2\Re\langle f_{\CB},f_{\CB'}\rangle.
$$
Since $\CB'$ refines $\CB$, $f_{\CB}$ is constant on each cell of $\CB'$. Hence $\langle f_{\CB}, f_{\CB'}\rangle = \langle f_{\CB}, f_{\CB}\rangle = \|f_{\CB}\|_2^2$. This concludes the proof.
\end{proof}

\begin{lemma}\label{lem:energy-incr}
Let $\CB$ be a $1$-factor of $G$ with complexity at most $M$ and growth $\CF$. Let $f: G\rightarrow \C$ be a $1$-bounded function with $\|f - f_{\CB}\|_{U^2(G)}\geq \delta$. Then there exists a growth function $\CF'$ depending only on $\CF,\delta$ and a refinement $\CB'$ of $\CB$ with complexity at most $2M$ and growth $\CF'$, such that 
$$
\|f_{\CB'}\|_2^2 - \|f_{\CB}\|_2^2 \gg \delta^8.
$$
\end{lemma}

\begin{proof}
By Lemma \ref{lem:u2-inverse-E} and enlarging $\CF$ if necessary, we can find a $1$-measurable subset $E \subset G$ with growth $\CF$ such that
$$
\Big|\E_{x \in G} (f - f_{\CB})(x) 1_E(x)\Big| \gg \delta^4.
$$
Let $\CB'$ be the refinement of $\CB$ whose cells are all of the form $B \cap E$ or $B \cap E^c$, where $B$ is a cell of $\CB$. By Lemma \ref{lem:intersect}, there exists a growth function $\CF'$ depending only on $\CF,\delta$ such that $\CB'$ has complexity at most $2M$ and growth $\CF'$. Since $1_E$ is constant on each cell of $\CB'$, we have
$$
\langle f-f_{\CB}, 1_E \rangle = \langle f_{\CB'} - f_{\CB}, 1_E\rangle.
$$
Hence
$$
\delta^4 \ll |\langle f_{\CB'}-f_{\CB}, 1_E\rangle| \leq \|f_{\CB'}-f_{\CB}\|_2,
$$
and the desired conclusion follows from Lemma \ref{lem:factor-pyth}.
\end{proof}

We are now ready to prove the weak regularity lemma by iterating Lemma \ref{lem:energy-incr}. The energy increment in the conclusion of Lemma \ref{lem:energy-incr} ensures that after a bounded number of iterations $f-f_{\CB}$ must have small $U^2$-norm.

\begin{proof}[Proof of Proposition \ref{prop:weak-reg}]
We set $\CB_0$ to be the trivial factor of $G$ consisting of only one cell (with complexity $M_0 = 1$ and growth $\CF_0$ for the constant function $\CF_0\equiv 1$), and iteratively define $\CB_1,\CB_2,\cdots$ as follows. For $k \geq 0$, suppose that $\CB_k$ is a $1$-factor of $G$ with complexity $M_k$ and growth $\CF_k$. If 
$$
\|f - f_{\CB_k}\|_{U^2(G)} \leq \delta,
$$
then we terminate the process. Otherwise, by Lemma \ref{lem:energy-incr}, there exists a growth function $\CF_{k+1}$ depending only on $\CF_k,\delta$ and a refinement $\CB_{k+1}$ of $\CB_k$ with complexity at most $M_{k+1} = 2M_k$ and growth $\CF_{k+1}$, such that
$$
\|f_{\CB_{k+1}}\|_2^2 - \|f_{\CB_k}\|_2^2 \gg \delta^8.
$$
This implies that the process must terminate at some $k \ll \delta^{-8}$, and the proof is completed by taking $\CB = \CB_k$.
\end{proof}

\section{A structure theorem for sumsets}\label{sec:sumset}

In this section we prove Proposition \ref{prop:sumset} and deduce Corollary \ref{cor:sumset} from it. To prove Proposition \ref{prop:sumset}, we use the regularity lemma to produce a $1$-measurable subset $A'$ and use the following counting lemma to deduce that $A+A$ is approximately the same as the popular sums in $A'+A'$.

\begin{lemma}\label{lem:counting}
Let $f,g: G\rightarrow \C$ be $1$-bounded functions such that $\|g\|_{U^2(G)} \leq \delta$ for some $\delta\in (0,1/2)$. Then $|(f*g)(x)| \leq \delta^{1/2}$ for all but at most $\delta |G|$ values of $x \in G$.
\end{lemma}

\begin{proof}
Since $\|g\|_{U^2(G)} \leq \delta$ we have $\|\wh{g}\|_{\infty} \leq \delta$.
Let $E$ be the exceptional set of $x \in G$ such that $|(f*g)(x)| > \delta^{1/2}$. Define $h: G\rightarrow\C$ to be the function supported on $E$, such that $|h(x)|=1$ and $\overline{h(x)}(f*g)(x) = |(f*g)(x)|$ for each $x \in E$. Then
$$
\langle f*g, h\rangle \geq \delta^{1/2} \frac{|E|}{|G|}.
$$
On the other hand, we have
$$
\langle f*g, h\rangle = \langle \wh{f}\cdot \wh{g}, \wh{h}\rangle \leq \|\wh{g}\|_{\infty} \|\wh{f}\|_2 \|\wh{h}\|_2 \leq \delta \Big(\frac{|E|}{|G|}\Big)^{1/2}.
$$
Combining the two inequalities above leads to $|E| \leq \delta |G|$ as desired.
\end{proof}

\begin{proof}[Proof of Proposition \ref{prop:sumset}]
Let $\delta>0$ be sufficiently small in terms of $\eps$.
By Proposition \ref{prop:weak-reg}, there exists a growth function $\CF$ depending only on $\delta$ and a $1$-factor $\CB$ of $G$ with complexity $O_{\delta}(1)$ and growth $\CF$, such that for $f = (1_A)_{\CB}$ we have
$$
\|1_A - f\|_{U^2(G)} \leq \delta.
$$
By Lemma \ref{lem:counting}, we have
$$
|(1_A*1_A)(x) - (f*f)(x)| \leq 2\delta^{1/2}
$$
for all but at most $2\delta |G|$ values of $x \in G$. Let $A' \subset G$ be the union of cells of $\CB$ on which $f$ takes values at least $\eps$. Then $A'$ is $1$-measurable with growth $\CF$ (after enlarging $\CF$ if necessary). If $x \in A\setminus A'$ then at most $\eps$-proportion of the elements in the cell of $\CB$ containing $x$ lies in $A$, and hence the number of such $x$ is at most $\eps |G|$. This establishes $|A\setminus A'| \leq \eps |G|$. Now let $x \in S_{\eps}(A')\setminus (A+A)$. Since $1_A*1_A(x)=0$, we have $(f*f)(x) \leq 2\delta^{1/2}$ for all but at most $2\delta |G|$ values of $x \in G$.
On the other hand, since $f$ takes values at least $\eps$ on $A'$, we have for $x \in S_{\eps}(A')$ the bound
$$
(f*f)(x) \geq \eps^2 \cdot \frac{|A' \cap (x-A')|}{|G|} \geq \eps^3.
$$
Once $\delta$ is chosen so that $2\delta^{1/2} < \eps^3$, this implies that $|S_{\eps}(A')\setminus (A+A)| \leq 2\delta |G| \leq \eps |G|$.
\end{proof}

The following lemma shows that the set of popular sums in $A+A$ is roughly the same as the full sumset after removing a small number of elements from $A$. It is a straightforward consequence of the arithmetic removal lemma in \cite{Green} or the almost-all version of the Balog-Szemer\'{e}di-Gowers theorem on restricted sumsets in \cite{Shao}.

\begin{lemma}\label{lem:popular-sum}
Let $A \subset G$ be a subset and let $\eps \in (0,1/2)$. Then for $\delta > 0$ sufficiently small in terms of $\eps$, there exists $A' \subset A$ with $|A\setminus A'| \leq \eps |G|$ such that the $\delta$-popular sumset
$$
S_{\delta}(A) = \{x \in A+A: |A \cap (x-A)| \geq \delta |G|\}
$$ 
satisfies $|(A'+A')\setminus S_{\delta}(A)| \leq \eps |G|$.
\end{lemma}

\begin{proof}
Let $X = (A+A)\setminus S_{\delta}(A)$. Then the number of solutions to $a_1+a_2=x$ with $a_1,a_2 \in A$ and $x \in X$ is at most $\delta |G||X| \leq \delta |G|^2$. Since $\delta$ is small enough in terms of $\eps$, the arithmetic removal lemma implies that one can remove at most $\eps |G|$ elements from $A, X$ to obtain $A',X'$, respectively, such that there are no solutions to $a_1+a_2=x$ with $a_1,a_2 \in A'$ and $x \in X'$. Hence $(A'+A') \cap X' = \emptyset$ and thus $|(A'+A') \cap X| \leq \eps |G|$. This implies that $|(A'+A')\setminus S_{\delta}(A)| \leq \eps |G|$.
\end{proof}

\begin{lemma}\label{lem:kneser}
Let $A \subset G$ be a subset. If $|A| > |G|/(\ell+1)$ for some positive integer $\ell$, then $|A+A| \geq |G|/\ell$.
\end{lemma}

\begin{proof}
By Kneser's theorem (see \cite[Theorem 5.5]{TaoVu}), we have
$$
|A+A| \geq 2|A+H| - |H|
$$
for some subgroup $H \subset G$. Let $k,m$ be positive integers defined by $k = [G:H]$ and $m = |A+H|/|H|$. Then
$$
|A| \leq |A+H| = m|H| = \frac{m}{k}|G|.
$$
Combining this with the hypothesis $|A| > |G|/(\ell+1)$, we obtain $k < m(\ell+1)$. Since $m\ell \geq m+\ell-1$, we have $m(\ell+1)-1 \leq (2m-1)\ell$, which implies that $k \leq (2m-1)\ell$. Hence
$$
|A+A| \geq (2m-1)|H| = \frac{2m-1}{k} |G| \geq \frac{|G|}{\ell}.
$$
This completes the proof.
\end{proof}

We end this section by recording the following corollary of Proposition \ref{prop:sumset}, which gives control on the size $A'$ in terms of the sumset $A+A$.

\begin{corollary}\label{cor:sumset}
Let $\eps \in (0,1/2)$. Let $A \subset G$ be a subset such that
$$
|A+A| \leq (\beta-\eps)|G|
$$
for some $\beta < 1/\ell$ and some positive integer $\ell$. There exists a growth function $\CF$ depending only on $\eps$ and a $1$-measurable subset $A' \subset G$ with growth $\CF$, such that $|A \setminus A'| \leq \eps |G|$ and 
$$
|A'| \leq \Big(\frac{1}{\ell+1}+\eps\Big)|G|.
$$ 
Moreover, if $G$ is a cyclic group of prime order, then we can ensure that 
$$
|A'| \leq \frac{1}{2}|A+A| + \eps |G| + 1.
$$ 
in the conclusion above.
\end{corollary}

\begin{proof}
Let $\delta > 0$ be sufficiently small in terms of $\eps$. By Proposition \ref{prop:sumset}, there exists a growth function $\CF$ depending only on $\delta$ and a $1$-measurable subset $A' \subset G$ with growth $\CF$, such that $|A\setminus A'| \leq \delta |G|$ and $|S_{\delta}(A') \setminus (A+A)| \leq \delta |G|$. By Lemma \ref{lem:popular-sum} applied to $S_{\delta}(A')$, we may find $A'' \subset A'$ with $|A'\setminus A''| \leq \eps |G|/2$ such that $|(A''+A'')\setminus S_{\delta}(A')| \leq \eps |G|/2$. It follows that
$$
|A''+A''| \leq |S_{\delta}(A')| + \frac{1}{2}\eps |G| \leq |A+A| + \eps |G| \leq \beta |G|.
$$
By Lemma \ref{lem:kneser}, we have $|A''| \leq |G|/(\ell+1)$, and hence
$$
|A'| \leq |A''| + \frac{1}{2}\eps |G| \leq \Big(\frac{1}{\ell+1}+\eps\Big)|G|.
$$
If $G$ is a cyclic group of prime order, then $|A''+A''| \geq 2|A''|-1$ and hence
$$
|A'| \leq \frac{1}{2}|A''+A''| + \frac{1}{2}\eps |G| + 1 \leq \frac{1}{2}|A+A| + \eps |G| + 1.
$$
This completes the proof.
\end{proof}

\section{Proof of Theorem \ref{thm:main}}\label{sec:deduction}

In this section we deduce Theorem \ref{thm:main} from Corollary \ref{cor:sumset}. We will apply Corollary \ref{cor:sumset} to both $G =\F_p$ as an additive group and $G = \F_p^{\times}$ as a multiplicative group. For this purpose, it is convenient to introduce the following notion.

\begin{definition}\label{def:sum-product}
Let $f: \F_p\rightarrow \C$ be $1$-bounded function. We say that $f$ is additively $1$-measurable with growth $\CF$ if it is $1$-measurable with growth $\CF$ when viewed as a function on the additive group $\F_p$. We say that $f$ is multiplicatively $1$-measurable with growth $\CF$ if it is $1$-measurable with growth $\CF$ when viewed as a function on the multiplicative group $\F_p^{\times}$ (ignoring the value $f(0)$). A subset $A \subset \F_p$ is said to be additively (resp. multiplicatively) $1$-measurable with growth $\CF$ if $1_A$ is additively (resp. multiplicatively) $1$-measurable with growth $\CF$.
\end{definition}

The characters on the additive group $\F_p$ are given by the additive characters $n\mapsto e_p(rn) := \exp(2\pi irn/p)$ for $r \in \F_p$. The characters on the multiplicative group $\F_p^{\times}$ are given by the Dirichlet characters $\chi\Mod{p}$. For a function $f: \F_p\rightarrow \C$, we use $\wh{f}(r)$  to denote the Fourier transform with respect to the additive characters and $\wh{f}(\chi)$ to denote the Fourier transform with respect to the multiplicative characters.

\begin{lemma}\label{lem:fg}
Let $f,g: \F_p\rightarrow \C$ be $1$-bounded functions such that
$$
\sum_{r \in \F_p} |\wh{f}(r)| \leq K\ \  \text{and}\ \  \sum_{\chi\Mod{p}} |\wh{g}(\chi)| \leq K
$$ 
for some $K \geq 2$.  Then
$$
\langle f, g\rangle = (\E f)\overline{(\E g)} + O(K^2p^{-1/2}).
$$
\end{lemma}

\begin{proof}
Using the additive and multiplicative Fourier expansions
$$
f(n) = \sum_{r \in \F_p} \wh{f}(r) e_p(rn) \ \ \text{and}\ \  g(n) = \sum_{\chi\Mod{p}} \wh{g}(\chi) \chi(n),
$$
we can write
$$
\langle f,g\rangle = \sum_{r \in \F_p} \sum_{\chi\Mod{p}} \wh{f}(r)\overline{\wh{g}(\chi)} \Big(\E_{n \in \F_p} e_p(rn) \overline{\chi(n)}\Big).
$$
By the Weil bound, the inner average over $n$ is $O(p^{-1/2})$ unless $\chi = \chi_0$ is trivial and $r = 0$. It follows that
$$
\langle f, g\rangle = \wh{f}(0)\overline{\wh{g}(\chi_0)} + O(K^2 p^{-1/2}).
$$
The desired conclusion follows since $\wh{f}(0) = \E f$ and $\wh{g}(\chi_0) = \E g + O(p^{-1})$.
\end{proof}

\begin{lemma}\label{lem:A-intersect-B}
Let $A ,B\subset \F_p$ be subsets. Suppose that $A$ is additively $1$-measurable with growth $\CF$ and $B$ is multiplicatively $1$-measurable with growth $\CF$. Then for any $\eps \in (0,1/2)$ we have
$$
\Big||A \cap B| - \frac{|A||B|}{p}\Big| \leq \eps p,
$$
provided that $p$ is sufficiently large in terms of $\CF,\eps$.
\end{lemma}

\begin{proof}
Let $M > 0$. We can find $1$-bounded functions $f,g: \F_p\rightarrow \C$ with $\|1_A - f\|_2 \leq M^{-1}$ and $\|1_B - g\|_2 \leq M^{-1}$, such that 
$$
\sum_r |\wh{f}(r)| \leq \CF(M)\ \  \text{and}\ \  \sum_{\chi\Mod{p}}|\wh{g}(\chi)|\leq \CF(M).
$$ 
Then $\E f = p^{-1}|A| + O(M^{-1})$, $\E g = p^{-1}|B| + O(M^{-1})$, and
$$
\frac{|A \cap B|}{p} = \langle 1_A, 1_B\rangle = \langle f, g\rangle + O(M^{-1}).
$$
By Lemma \ref{lem:fg}, we have
$$
\langle f,g\rangle = (\E f)\overline{(\E g)} + O(\CF(M)^2 p^{-1/2}) = \frac{|A||B|}{p^2} + O(M^{-1} + \CF(M)^2p^{-1/2}).
$$
Combining the estimates above leads to the conclusion, provided that $M$ is chosen to be large enough in terms of $\eps$ and $p$ is large enough in terms of $\CF,\eps$.
\end{proof}

\begin{proof}[Proof of Theorem \ref{thm:main}]
Let $A \subset \F_p$ be a subset such that
$$
\max(|A+A|, |A\cdot A|) \leq (\beta - \eps)p,
$$
for some $\beta < 1/\ell$ and some positive integer $\ell$. By Corollary \ref{cor:sumset} (applied first to $A$ as a subset of the additive group $\F_p$ and then to $A$ as a subset of the multiplicative group $\F_p^{\times}$), we may find a growth function $\CF$ depending only on $\eps$, such that there is an additively $1$-measurable subset $A_1 \subset \F_p$ with growth $\CF$ and a multiplicatively $1$-measurable subset $A_2 \subset \F_p$ with growth $\CF$ such that $|A\setminus A_1| \leq \eps p/5$, $|A\setminus A_2| \leq \eps p/5$, and
$$
|A_1| \leq \frac{1}{2}|A+A| + \frac{\eps}{5} p+1 \leq \Big(\frac{\beta}{2}-\frac{\eps}{4}\Big)p, \ \ |A_2| \leq \Big(\frac{1}{\ell+1} + \frac{\eps}{5}\Big)p.
$$
By Lemma \ref{lem:A-intersect-B} we have
$$
|A_1 \cap A_2| \leq \frac{|A_1||A_2|}{p} + \frac{\eps^2}{20} p \leq \Big(\frac{\beta}{2(\ell+1)} + \frac{\eps}{10}\Big)p.
$$
It follows that
$$
|A| \leq |A_1\cap A_2| + |A\setminus A_1| + |A\setminus A_2| \leq \Big(\frac{\beta}{2(\ell+1)} + \eps\Big)p.
$$
This is almost the desired upper bound for $|A|$, apart from the extra $\eps p$ term. This can be removed by running the argument above with $\beta$ replaced by $\beta - \eps /2$ and $\eps$ replaced by $\eps/(100\ell)$.
\end{proof}

\bibliographystyle{plain}
\bibliography{biblio}

\end{document}